\documentclass[12pt]{article}
\usepackage{amssymb,amsmath,amsthm}

\title{Elementary equivalence \\ of stable linear groups over fields of characteristic 2 }
\author{E. I. Bunina\footnote{Bar Ilan University (Israel), Department of mathematics}, A. V. Mikhalev, I. O. Solovyev\footnote{M.V. Lomonosov Moscow State University,
Faculty of mechanics and mathematics}\\}

\newtheorem{lem}{Lemma}
\newtheorem{theo}{Theorem}
\newtheorem{theorem}{Theorem}

\newtheoremstyle{neosn}{0.5\topsep}{0.5\topsep}{\rm}{}{\bf}{.}{ }{\thmname{#1}\thmnumber{ #2}\thmnote{ {\mdseries#3}}}
\theoremstyle{neosn}
\newtheorem{definition}{Definition}

\usepackage{indentfirst}

\newcommand{\diag}{\mathrm{diag}\,}

\newcommand{\GL}{\mathrm{GL}\,}

\newcommand{\SL}{\mathrm{SL}\,}

\newcommand{\Mat}{\mathrm{Mat}\,}
\newcommand{\PSL}{\mathrm{PSL}\,}

\newcommand{\PGL} {\mathrm{PGL}\,}

\def\q#1.{{\bf #1.}}

\begin{document}
\maketitle
In this paper, we prove a criterion of elementary equivalence of stable linear groups over fields of characteristic two.

\section{Introduction, History and Definitions}\

\subsection{Elementary equivalence.}

Two structures of the same signature are called \emph{elementary equivalent} if they satisfy the same first order sentences in their signature.
Any two finite structures with the same signature are elementarily equivalent if and only if they are isomorphic.
Any two isomorphic structures are elementarily equivalent, but the opposite is not always true.
For example, the field of complex numbers $\mathbb C$ and field of algebraic numbers $\overline{\mathbb Q}$ are elementarily equivalent, but they cannot be isomorphic due to difference in their cardinalities.

Tarski and Maltsev pioneered the theory of describing groups and rings from an elementarily equivalence standpoint.
Several complete results were obtained: for example, two algebraically closed fields are elementary equivalent if and only if they have the same characteristics;
two Abellian groups are elementary equivalent if and only if they have the same special ``characteristic numbers'' (Szmielew, \cite{Shmelew});
similar results with invariants were obtained for Boolean rings (Ershov--Tarski, \cite{Boolean}).

An outstanding result was the answer for the old problem raised by A.\,Tarski around 1945:
for free groups the elementary theory doesn't distinguish these groups (see the series of works of Kharlampovich--Myasnikov and Z.\,Sela, e.\,g. \cite{Kharlamp-Myasikov},~\cite{Sela}). The similar situation takes place for the torsion free hyperbolic groups (see Sela,~\cite{Sela2}).

\subsection{Maltsev-type theorems for linear groups.}

It is also interesting to study connections between logic properties of some basic structures and logic properties of structures derived from these basic structures.

First results of such type was obtained by Maltsev in 1961 in~\cite{Maltsev}.
He proved that the groups 
$\mathcal G_n(K_1)$ and $\mathcal G_m(K_2)$ (where $G=\GL,\SL,\PGL, \PSL$, $n,m\geqslant3$, $K_1$,~$K_2$
are fields of characteristics~$0$) are elementarily equivalent if and only if $m=n$ and the fields  $K_1$~and~$K_2$ are elementarily equivalent.

This type of correspondence are often called  \emph{Maltsev translation}.
It means that logic properties are completely translated from basic structures to derived structures and vice-versa.

In 1961--1971 Keisler (\cite{Keisler}) and Shelah (\cite{Shelah}) proved the next important Isomorphism theorem:

\begin{theorem}
Two models  $\mathcal U_1$ and $\mathcal U_2$ are elementarily equivalent if and only if
there exists an ultrafilter  $\mathcal F$ such that their ultrapowers coincide:
$$
\prod_{\mathcal F} \mathcal U_1\cong \prod_{\mathcal F} \mathcal U_2.
$$
\end{theorem}

This theorem allowed Beidar and Mikhalev in 1992  (see~\cite{BeidMikh1}) to generalize Maltsev theorem for the case when  $K_1$ and $K_2$ are skewfields and prime associative rings.  
This approach for the groups $\GL_n$ was generalized in~\cite{Bunina_Bragin}  by the following result:

\medskip

\begin{theorem}\label{finite number idempotents}
Let $R_1$ and $R_2$ be associative rings with~$1$
$(1/2)$ with finite number of central idempotents and $m,n\geqslant 4$ $(m,n\geqslant 3)$. Then $\GL_m(R_1)\equiv \GL_n(R_2)$ if and only if there exist  central idempotents $e\in R$  and $f \in S$ such  that $e M_m(R)\equiv f M_n(S)$ and $(1-e)M_m(R)\equiv (1-f) M_n(S)^{op}$.
\end{theorem} 

\medskip

Continuation of investigations in this field were the papers of Bunina 1998--2010.
Similar to Maltsev's results were obtained not only for classical linear groups $\GL, \PGL, \SL, \PSL$, but also for unitary linear groups over fields, skewfields, and rings with involutions (see  \cite{Bunina_intro2}, \cite{Bunina_intro1}), for Chevalley groups over fields (\cite{Bunina_intro3}), over local rings (see \cite{Bunina_intro4}) and arbitrary commutative rings (see~\cite{Bun_Chevalley_arb}), and also for other different derivative structures.

In some cases elementary equivalence of derivative structures (even a little similar to linear groups) is equivalent not to elementary but to more strong equivalence of initial structures. Often second-order equivalence or some of its limitations can appear.

For example, in 2000, V. Tolstykh \cite{Tolstyh} stated the connection between second-order properties of skewfields and first-order properties of automorphism groups of infinite-dimensional linear spaces over them. In 2003, E. I. Bunina and A. V. Mikhalev (see~\cite{categories}) stated the connection between second-order properties of associative rings and elementary properties of categories of modules, endomorphism rings, automorphism groups, and projective spaces of infinite rank over these rings.
Similar results were obtained also for endomorphism rings and automorphism groups of Abelian   $p$-groups: Bunina, Mikhalev and Roizner (see~\cite{BunMikhRoizner}) proved that two endomorphism rings of Abelian  $p$-groups or two automorphism groups  of Abelian $p$-groups for $p\geqslant 3$ are elementarily equivalent  if and only if initial Abelian groups are equivalent in the full second order logic (in one exceptional case in its limitation).

\medskip

\subsection{Stable Linear Groups}
Let $R$ be an associative ring with unit. The following definitions correspond to~\cite{HahnOMeara}.
\begin{definition}
Denote by $\Mat_{\infty}(R)$ the ring of matrices with countable number of rows and columns such that out of the main diagonal there are only a finite number of nonzero elements, and also there exists a number $n$ such that for any $i \geqslant n$ the elements $r_{ii} = a$, $a\in R$.
\end{definition}
It is clear that $\Mat_{\infty}(R)$ is a ring.

Let $A\in \GL_n(R)$. We identify $A$ with an element from $\GL_n(R)$ by the following rule: $A$ is written in the left upper corner, starting from the position $(n,n)$ the diagonal contains $1$, and all other places contain zeros.

We preserve the notation $\GL_n(R)$ for the obtained subgroups $\Mat_{\infty}(R)$. It is clear that $\GL_n(R)$ are subgroups of the groups of invertible elements of the ring $\Mat_{\infty}(R)$, and also it is clear that for $m \geq n$ we have $\GL_n(R) \subseteq \GL_m(R)$.

\begin{definition}
Let us set
$$
\GL(R) = \bigcup\limits_{n \geqslant 1} \GL_n(R)\qquad 
(\GL_n(R) \subseteq \Mat_{\infty}(R)).
$$
It is a subgroup of the group of invertible elements of the ring $\Mat_{\infty}(R)$. Let us call it the \emph{stable linear group}.
\end{definition}

For stable linear groups their automorphisms were described: Atkarkaya described automorphisms of stable linear groups E(R) and GL(R) over commutative local rings $R$ with 1/2 (see~\cite{atkarskaya}).

In our previous paper (see~\cite{bunmikhsol}) it was proved  that despite an ``infinite'' dimension of the stable group, from elementary equivalence of two arbitrary rings with unit it follows elementary equivalence of stable linear groups over them, i.e., we do not need higher-order logic.  In the same paper we proved that from elementary equivalence of stable linear groups over commutative local rings with 1/2 it follows elementary equivalence of the corresponding rings.

In the given work we  extend this previous result to the stable linear groups over fields of characteristic~$2$.

\section{Proof of the main theorem}

Consider the group $\GL(R)$, where $R$ is a field of characteristic 2.
Let us denote its unit by~$E$.
For more simple formulas we will write $A \sim B$ instead of $\exists U\ A=UBU^{-1}$  and call  $A$ to be   \emph{conjugate} to $B$.

Our first goal is to define elementarily a subgroup isomorphic to the group $\GL_2(R)$. We will use matrices $\diag[1, 1, T]$, where
$$T \sim
\diag\left[
\begin{pmatrix}
1 & 1\\
1 & 0
\end{pmatrix},
\dots
\right]$$

The first part of the proof depends on existence of third roots of unity in the field.

\subsection{Fields containing third roots of unity}
Let $K$ be a field of characteristic $2$ containing third roots of unity $1, \xi, \xi^2$.

The following lemma is proved, for example, in~\cite{kaleeva}; it follows from classical linear algebra results.
\begin{lem}\label{trans}
For any  number of mutually commuting third-order elements of the group $\GL_n(K)$, there exists a basis in which all of them have diagonal form with third roots of unity on the diagonal. 

It is clear that for a finite set of commuting matrices the same condition holds also in the stable linear group $\GL(K)$.
\end{lem}

Let $A$ be a set of matrices. Denote by $|||A|||$ the number of distinct conjugacy classes of $A$.

\begin{lem}
Let us consider  a set $D=[d_1, d_2, d_3, \ldots, d_k]$ of elements of $K$. If $D$ includes $[\xi, \xi]$ or $[\xi, \xi^2]$, then $||| \{AB :A\sim \diag[D], B \sim \diag[D], AB = BA \} ||| >2$.
\end{lem}

\begin{proof}
To prove this lemma it is sufficient to consider only one pair of elements $d_i$, since other elements can be just fixed at the first positions. Also it is sufficient to consider only permutation of diagonal elements.

First consider the case when $D$ contains $[\xi, \xi]$.
We assume that these elements are corresponded to $d_{k-1}, d_k$. The first $k-2$ elements are fixed, under multiplication they give the same set of eigenvalues. 
We can obtain the following matrices: 
\begin{align*}
\diag[1,1,\xi,\xi] \cdot \diag[\xi,\xi,1,1]& = \diag[\xi,\xi,\xi,\xi],\\
\diag[1,1,\xi,\xi] \cdot \diag[\xi,1,\xi,1] &= \diag[\xi,1,\xi^2,\xi],\\
\diag[1,1,\xi,\xi] \cdot \diag[1,1,\xi,\xi] &= \diag[1,1,\xi^2,\xi^2].
\end{align*}
We omitted for simplicity of reading elements $d_1, d_2, d_3, \ldots, d_{k-2}$ and final units of stable matrices.
All these  matrices have different sets of eigenvalues, so they are pairwise non-conjugate.

If $D$ contains $[\xi, \xi^2]$, then 
\begin{align*}
\diag[1,1,\xi,\xi^2] \cdot \diag[1,1,\xi,\xi^2]& = \diag[1,1,\xi^2,\xi],\\
\diag[1,1,\xi,\xi^2] \cdot \diag[\xi,\xi^2,1,1] &= \diag[\xi,\xi^2,\xi,\xi^2],\\
\diag[1,1,\xi,\xi^2] \cdot \diag[1,\xi^2,\xi,1] &= \diag[1,\xi^2,\xi^2,\xi^2],
\end{align*}
the same situation.
\end{proof}

As a result of these two lemmas we obtain the following lemma:

\begin{lem}
Let $A^3 = E$. If $||| \{BC :B\sim A, C \sim A, CB = BC \} ||| = 2$, then $A \sim \diag[\xi]$ (or $A \sim \diag[\xi^2]$).
\end{lem}

This lemma gives a method how to elementarily define a matrix conjugated to one of $\diag[\xi]$ and $\diag[\xi^2]$.
\begin{equation}
\begin{split}
\varphi(A) :=\ & \exists X_1 \exists X_2 \exists Y_1 \exists Y_2 \forall Z_1 \forall Z_2\  
(A^3 = E) \land \lnot(A = E) \land\\
&\land (X_1 \sim X_2 \sim Y_1 \sim Y_2 \sim A)\land(X_1X_2 = X_2X_1) \land (Y_1Y_2=Y_2Y_1) \land\\
&\land \Big(\big((Z_1 \sim A) \land (Z_2 \sim A) \land (Z_1Z_2=Z_2Z_1)\big) \rightarrow \big((Z_1Z_2 \sim Y_1Y_2)\lor(Z_1Z_2 \sim X_1X_2)\big)\Big)
\end{split}
\end{equation}
Suppose this formula holds for $A$. Assume that $A \sim \diag[\xi]$. It is possible to obtain according to a replacement of notations $\xi^2 \rightarrow \xi', \xi \rightarrow \xi'^2$. Then $A^2 \sim \diag[\xi^2]$.

Consider the formula
$$
\psi(B) := \exists X_1 \exists X_2 (X_1 \sim A)\land(X_2 \sim A^2) \land (X_1X_2=X_2X_1) \land(X_1X_2 = B) \land \lnot\varphi(B).
$$
If this formula holds for $B$, then $B \sim \diag[\xi, \xi^2]$.

Let $X_1 \sim X_2 \sim A;\ X_1X_2=X_2X_1$. We can suppose that we chose a basis in such that $X_1 = \diag[\xi, 1, \ldots], X_2 = \diag[1, \xi, \ldots]$.

Let us consider the formula
$$
\theta(C) := (CX_1=X_1C)\land (CX_2=X_2C)\land(C\sim B)
\land \bigwedge_i \lnot\varphi(BX_i)
\land \bigwedge_i \lnot\varphi(BX_i^2 ).
$$

If this formula holds for $C = (c_{ij})$, then since $c$ commutes with $X_i, i=1,2$, we have $c_{1j}=c_{i1} = 0, i,j>1; c_{2j}=c_{i2} = 0, i,j>2$. $C$ has order~$3$, so $c_{11}^3=1,c_{22}^3=1$. Suppose that $c_{11} \neq 1$. Then either $c_{11} = \xi$, or $c_{11} = \xi^2$. Suppose (without loss of generality) that $c_{11} = \xi$.
Then, formula $\varphi$ holds for $CX_1$, it contradicts  to  definition of $\theta$. Therefore $c_{11} = 1$, and similarly $c_{22} = 1$.
Then $C = \diag[1,1,T],$  where $T \sim \diag[\xi, \xi^2]$.

\begin{lem}\label{sim}
$\diag[\xi, \xi^2] \sim 
\begin{pmatrix}
1 & 1\\
1 & 0
\end{pmatrix}.
$
\end{lem}

\begin{proof}
To prove this lemma we need to find the eigenvalues of the matrix.   Its characteristic polynomial is $P(\lambda)= \lambda^2 + \lambda + 1$. Substituting $\xi$ to the characteristic polynomial, we obtain $P(\xi) = \xi^2+\xi+1$.  Note that $P(\xi)=\xi P(\xi)$, hence
$$
P(\xi)+\xi P(\xi)=0 \Rightarrow P(\xi)(1+\xi) =0\Rightarrow P(\xi)=0,
$$
Similarly for $\xi^2$. Therefore $\xi$ (and $\xi^2$) are roots of the characteristic polynomial.
\end{proof}

\subsection{Fields without third roots of unity}

Let $K$ be a field of characteristic 2 without non-trivial roots of the third power of unity.

\begin{lem}\label{trans}
For any finite number of mutually commuting  elements of the order $3$ in $\GL_n(K)$, there exists a basis in which all of them have block-diagonal form with $E_{11} + E_{12}+E_{21}, E_{22} + E_{12}+E_{21}$ and $E_{11} + E_{22}$ on their diagonals.
\end{lem}
This lemma is proved in \cite{kaleeva} for the elements from $\GL_n$, but it is obviously correct for the group $\GL$  too.

Denote  $T = E_{11} + E_{12}+E_{21}$ and $D_k = \diag [\underbrace{T,T,\ldots,T,T}_{k},1, \ldots]$.

\begin{lem}
If $k>1$, then $||| \{AB : A\sim D_k, B \sim D_k, AB = BA, A, B\in \GL(K)\} ||| >3$.
\end{lem}
\begin{proof}
It is sufficient  to consider only two  blocks $T$, which exist in the block-diagonal form of~$A$.
By permutations of these two blocks we can obtain
\begin{align*}
\diag[E,E,T,T] \cdot \diag[T, T,E,E]& = \diag[T, T, T, T],\\
\diag[E,E,T,T] \cdot \diag[E,E,T^2,T^2] &= \diag[E, E, E, E],\\
\diag[E,E,T,T] \cdot \diag[E,E,T, T] &= \diag[E, E, T^2, T^2],\\
\diag[E,E,T^2,T] \cdot \diag[E, T,T^2,E] &= \diag[E, T, T, T],
\end{align*}
where $E$ is a unity matrix of order $2\times 2$. As it follows from Lemma~\ref{sim}, $T^2 \sim T$.
\end{proof}

This lemma gives us a formula that holds for matrix $A$ if and only if $A\sim D_1$:
\begin{equation}
\begin{split}
\varphi'(A) :=\ & \exists X_1 \exists X_2 \exists Y_1 \exists Y_2 \forall Z_1 \forall Z_2\  
(A^3 = E) \land \lnot(A = E)\ \land\\
&\land (X_1 \sim X_2 \sim Y_1 \sim Y_2 \sim A)\land(X_1X_2 = X_2X_1) \land (Y_1Y_2=Y_2Y_1)\ \land\\
&\land \Big(\big((Z_1 \sim A) \land (Z_2 \sim A) \land (Z_1Z_2=Z_2Z_1)\big) \rightarrow\\
& \rightarrow\big((Z_1Z_2 \sim X_1X_2)\lor(Z_1Z_2 \sim Y_1Y_2)\lor(Z_1Z_2 \sim E)\big)\Big).
\end{split}
\end{equation}

\begin{lem}\label{commi}
Let $A= (a_{ij})$ commute with
$$
G_k = \diag[\underbrace{1,1,\ldots, 1}_{k-1}, T, 1,1 \ldots],\text{ where }k \geqslant 1.
$$
Then
\begin{enumerate}
  \item \label{null_col} If $j + 1 < k$ or $k + 1 < j$, then
  $$
  \begin{pmatrix}
  a_{j,k} & a_{j,k+1}\\
  a_{j+1, k} & a_{j+1, k+1}
  \end{pmatrix}
  =
  \begin{pmatrix}
  0 & 0\\
  0 & 0
  \end{pmatrix}
  $$
  \item \label{null_row} If $j + 1 < k$ or $k + 1 < j$, then
  $$
  \begin{pmatrix}
  a_{k,j} & a_{k,j+1}\\
  a_{k+1, j} & a_{k+1, j+1}
  \end{pmatrix}
  =
  \begin{pmatrix}
  0 & 0\\
  0 & 0
  \end{pmatrix}
  $$
\item \label{diag_comm}
$$
\begin{pmatrix}
a_{k,k} & a_{k,k+1}\\
a_{k+1, k} & a_{k+1, k+1}
\end{pmatrix}
=
\begin{pmatrix}
a & b\\
b & a+b
\end{pmatrix},
$$
for some $a,b$.
\end{enumerate}
\end{lem}

\begin{proof}
(1) Assume that $A$ commutes with $G_k$ and $B=(b_{ij}) = AG_k= G_k A$. Consider
$$
\begin{pmatrix}
b_{j,k} & b_{j,k+1}\\
b_{j+1, k} & b_{j+1, k+1}
\end{pmatrix}.
$$
When $A$ is  multiplied by $G_k$ from the left:
$$
\begin{pmatrix}
b_{j,k} & b_{j,k+1}\\
b_{j+1, k} & b_{j+1, k+1}
\end{pmatrix}
=
\begin{pmatrix}
a_{j,k} & a_{j,k+1}\\
a_{j+1, k} & a_{j+1, k+1}
\end{pmatrix}.
$$
When $A$  multiplied by $G_k$ from the right:
$$
\begin{pmatrix}
b_{j,k} & b_{j,k+1}\\
b_{j+1, k} & b_{j+1, k+1}
\end{pmatrix}
=
\begin{pmatrix}
a_{j,k} +  a_{j,k+1} & a_{j,k}\\
a_{j+1, k} + a_{j+1, k+1} & a_{j+1, k}
\end{pmatrix}.
$$
Then:
$$
\begin{pmatrix}
a_{j,k} & a_{j,k+1}\\
a_{j+1, k} & a_{j+1, k+1}
\end{pmatrix}
=
\begin{pmatrix}
a_{j,k} +  a_{j,k+1} & a_{j,k}\\
a_{j+1, k} + a_{j+1, k+1} & a_{j+1, k}
\end{pmatrix}.
$$
From the equality of left columns of both matrices we obtain  $a_{j,k+1} = a_{j+1, k+1}  = 0$. But then from the equality of right columns of both matrices we obtain that all other elements equal to zero too. The proposition (2) is proved in exactly the same way as proposition (1).

(3) We have
$$
\begin{pmatrix}
a_{k,k}+ a_{k+1, k} & a_{k,k+1}+ a_{k+1, k+1}\\
a_{k,k} & a_{k,k+1}
\end{pmatrix}
=
\begin{pmatrix}
a_{k,k} + a_{k,k+1} & a_{k,k}\\
a_{k+1, k} + a_{k+1, k+1} & a_{k+1, k}
\end{pmatrix}.
$$
Define $b:=a_{k,k+1} = a_{k+1,k}$, $a:= a_{kk}$, then $a_{k+1, k+1} = a + b$.
\end{proof}

Let $X$ be a matrix satisfied the formula $\varphi'$. Consider the formula
$$
\theta'(C) := (CX=XC)\land(C\sim X)
\land (CX \neq E)
\land (CX^2 \neq E).
$$
Let $C$ be a matrix such that $\theta'(C)$ is true. Let us fix a basis such that $X$ has a form $G_1$. Then from Lemma~\ref{commi} we have $c_{1j}=c_{i1} = 0, c_{2j}=c_{i2} = 0, i,j>2$. Also we have
$$
\begin{pmatrix}
c_{11} & c_{12}\\
c_{21} & c_{22}
\end{pmatrix}
=
\begin{pmatrix}
a & b\\
b & a+b
\end{pmatrix}
$$
and
$$
\begin{pmatrix}
a & b\\
b & a+b
\end{pmatrix}^3 = 
\begin{pmatrix}
1 & 0\\
0 & 1
\end{pmatrix}
$$
or
$$
\begin{pmatrix}
a^3 + ab^2 + b^3 & a^2b + ab^2\\
a^2b + ab^2 & a^3 + a^2b + b^3
\end{pmatrix} = 
\begin{pmatrix}
1 & 0\\
0 & 1
\end{pmatrix}.
$$
It follows that  $a^2b + ab^2 = ab(a+b) = 0$. There is no zero divisors in~$K$, therefore there are three possible cases:
\begin{enumerate}
  \item $b = 0$. Then $a^3 = 1 \Rightarrow a = 1$.
  \item $a = 0$. Then $b^3 = 1 \Rightarrow b = 1$.
  \item $a = b$. Then $ a^3 + ab^2 + b^3 = 1 \Rightarrow a^3 + a^3 + a^3 = 1 \Rightarrow a^3 = 1 \Rightarrow a = b = 1$.
\end{enumerate}
The first case corresponds to the unity matrix. The second and the third cases correspond to $
\begin{pmatrix}
c_{11} & c_{12}\\
c_{21} & c_{22}
\end{pmatrix}
=T$
and
$
\begin{pmatrix}
c_{11} & c_{12}\\
c_{21} & c_{22}
\end{pmatrix}
=T^2$.
Both thess cases are impossible due to the formula $\theta'$, because such $C$  multiplied by $X$ or $X^2$ gives the unity matrix.

So we peoved that $C$ has the form $\diag[1,1,T'],$  where  $T' \sim D_1$.

\subsection{Elementary definability of $\GL_2(K)$ and the main theorem}

Assume the formula $\theta$ holds only for matrices $\diag[1,1,T],$  where  $T \sim
\begin{pmatrix}
1 & 1\\
1 & 0
\end{pmatrix}$. Let us determine $\GL_2(K)$ using these matrices.

Let us notice that the formula $\theta$ holds for $G_k, k > 2$. If  $M$ commutes with $G_{2k+1}, k = 1,2,3, \ldots$, then $M = \diag[M_0, M_1, M_2, M_3, M_4, \ldots, M_k]$, where $M_0 \in \GL_2(K)$, and $M_i =
\begin{pmatrix}
a_i & b_i\\
b_i & a_i+b_i
\end{pmatrix}, i >0
$.

If $M$ commutes also with $G_{2k + 2}, k = 1,2,3, \ldots$,
then $M = \diag[M_0, a_1, M'_1, M'_2, M'_3, M'_4, \ldots, M'_{k'}]$, where $M'_i  =
\begin{pmatrix}
a'_i & b'_i\\
b'_i & a'_i+b'`_i
\end{pmatrix}
$.
Comparing  block structures of different  representations of $M$, we obtain
$$
M = \diag[M_0, M_1, M_2, M_3, M_4, \ldots, M_k] = \diag[M_0, a_1, M'_1, M'_2, M'_3, M'_4, \ldots, M'_{k'}].
$$
Therefore
$b_i = b'_i = 0$, $a_i = a'_i$ and $a'_i = a_{i+1}$. Hence $a_i = a_{i+1}$ for every $i$. It follows from the definition of the stable linear group that there is only a finite number of non-unit elements on the diagonal. Hence $a_i = 1$.

Therefore, we obtain a formula that defines a subgroup isomorphic to $\GL_2(K)$:
$$
\gamma(M) = \forall C \Big(\theta(C)\rightarrow (MC=CM)\Big).
$$

Now we are ready to prove the main theorem of this paper.

\begin{theo}
Let $K_1$ and $K_2$ be fields of characteristic 2. If the stable linear groups  $\GL(K_1)$ and $\GL(K_2)$  are elementarily equivalent, then the fields $K_1$ and $K_2$  are also elementarily equivalent.
\end{theo} 

\begin{proof}
Let $\GL(K_1) \equiv \GL(K_2)$. Since in the previous lemmas we proved that the subgroup $\GL_2$ is elementarily defined in $\GL$ for a field of characteristic~$2$, this means that $\GL_2(K_1)\equiv \GL_2(K_2)$. 
By the generalization of  Maltsev Theorem proved for example in~\cite{bunina}  this implies $K_1\equiv K_2$.
\end{proof}

\end{document}